\pdfoutput=1
\RequirePackage{ifpdf}
\ifpdf % We~are running pdfTeX in pdf mode
\documentclass[pdftex]{sigma}%,draft
\else
\documentclass{sigma}
\fi

\numberwithin{equation}{section}

\newtheorem{Theorem}{Theorem}[section]
\newtheorem{Corollary}[Theorem]{Corollary}
\newtheorem{Lemma}[Theorem]{Lemma}
 { \theoremstyle{definition}
\newtheorem{Definition}[Theorem]{Definition}
\newtheorem{Remark}[Theorem]{Remark} }

\def\C{{\mathbb C}}
\def\R{{\mathbb R}}
\def\Z{{\mathbb Z}}
\def\Q{{\mathbb Q}}

\def\D{D}

\def\ph{\widehat{p}}

\def\Mgno{\overline{\mathcal Q}_{g,n}}
\def\M{{\mathcal M}}
\def\Mb{\overline{{\mathcal M}}}

\def\f{\frac}
\def\e{\epsilon}

\def\a{\alpha}
\def\b{\beta}
\def\s{\sigma}
\def\deg{{\rm deg}}

\def\Ch{\widehat{C}}
\def\gh{\widehat{g}}
\def\Bh{{\widehat B}}
\def\l{\lambda}

\def\wO{{\widehat{\Omega}}}

\def\pt{\lambda_{\rm P}}
\def\at{\tilde{a}}
\def\bt{\tilde{b}}
\def\ut{\tilde{u}}

\def\bu{{\bf u}}

\def\Bh{{\widehat{B}}}

\def\bu{\pmb{u}}

\def\qb{{\xi}}

\def\gh{\widehat{g}}

\def\Mgn{\mathcal{Q}_{g,n}}
\def\Mgno{\overline{\mathcal{Q}}_{g,n}}

\def\Cb{\overline{{\mathcal C}}}
\def\Qc{{\mathcal Q}}

\begin{document}
\allowdisplaybreaks

\renewcommand{\thefootnote}{}

\newcommand{\arXivNumber}{2108.01419}

\renewcommand{\PaperNumber}{001}

\FirstPageHeading

\ShortArticleName{Tau Function and Moduli of Meromorphic Quadratic Differentials}

\ArticleName{Tau Function and Moduli\\ of Meromorphic Quadratic Differentials\footnote{This paper is a~contribution to the Special Issue on Mathematics of Integrable Systems: Classical and Quantum in honor of Leon Takhtajan.

~~\,The full collection is available at \href{https://www.emis.de/journals/SIGMA/Takhtajan.html}{https://www.emis.de/journals/SIGMA/Takhtajan.html}}}

\Author{Dmitry KOROTKIN~$^{\rm ab}$ and Peter ZOGRAF~$^{\rm bc}$}

\AuthorNameForHeading{D.~Korotkin and P.~Zograf}

\Address{$^{\rm a)}$~Department of Mathematics and Statistics, Concordia University,\\
\hphantom{$^{\rm a)}$}~1455 de Maisonneuve West, Montreal, H3G 1M8 Quebec, Canada}
\EmailD{\href{mailto:dmitry.korotkin@concordia.ca}{dmitry.korotkin@concordia.ca}}

\Address{$^{\rm b)}$~Euler International Mathematical Institute, Pesochnaja nab.~10,\\
\hphantom{$^{\rm b)}$}~Saint Petersburg, 197022 Russia}
\EmailD{\href{zograf@pdmi.ras.ru}{zograf@pdmi.ras.ru}}

\Address{$^{\rm c)}$~Chebyshev Laboratory, St.~Petersburg State University, 14th Line V.O.~29,\\
\hphantom{$^{\rm c)}$}~Saint Petersburg, 199178 Russia}

\ArticleDates{Received August 09, 2021, in final form December 28, 2021; Published online January 03, 2022}

\Abstract{The Bergman tau functions are applied to  the study of the  Picard group of moduli spaces of  quadratic differentials with at most $n$ simple poles on genus~$g$ complex algebraic curves. This generalizes our previous results on moduli spaces of holomorphic quadratic differentials.}


\Classification{14H15; 14H70; 14K20; 30F30}

\renewcommand{\thefootnote}{\arabic{footnote}}
\setcounter{footnote}{0}

\begin{flushright}
\begin{minipage}{53mm}
\it To Leon Takhtajan on occasion\\ of his 70$\,{}^{th}$ birthday
\end{minipage}
\end{flushright}

\section{Introduction and statement of results}

The theory of the Bergman tau function was  applied to the study of geometry of various moduli spaces in
\cite{CMP, Adv,Annalen,MRL,Contemp}; see \cite{review} for a review of known results.

The present paper is a continuation of \cite{Contemp}. Here we use Bergman tau functions to study the geometry of the moduli space $\Mgn$ of quadratic differentials with $n$ simple poles on genus~$g$ complex algebraic curves. More precisely, the space $\Mgn$ is defined as the set of isomorphism classes of pairs $(C, q)$, where~$C$ is a smooth genus~$g$ complex curve with~$n$ labeled distinct marked points, and $q$ is a meromorphic quadratic differential on~$C$ with at most simple poles at the marked points and no other poles (throughout the paper we will assume that $2g+n>3$). It is well known that $\Mgn$ is naturally isomorphic to $T^*\M_{g,n}$, the total space of the holomorphic cotangent bundle on the moduli space $\M_{g,n}$ of pointed complex algebraic curves.

The bundle $T^*\M_{g,n}$ can be extended to the Deligne--Mumford boundary of $\M_{g,n}$ in two natural ways: first, as the cotangent bundle $T^*\Mb_{g,n}$ on $\Mb_{g,n}$, the Deligne--Mumford moduli space of stable curves, and second, as the total space of the direct image $\pi_*\omega_{g,n}^2$, where $\omega_{g,n}$ is the relative dualizing sheaf on the universal curve $\pi\colon \Cb_{g,n}\to\Mb_{g,n}$. These two extensions are rather close to each other: namely, $\det T^*\Mb_{g,n} = \det\pi_*\omega_{g,n}^2-\delta_{\rm DM}$, where $\delta_{\rm DM}=\Mb_{g,n}\setminus\M_{g,n}$ is the Deligne--Mumford boundary class (a detailed analytic treatment of these subjects can be found in~\cite{Wo}). In this paper we will use the second extension and put $\Mgno=\pi_*\omega_{g,n}^2$.

The moduli space $\Mgno$ has an open dense subset $\Qc\big(1^{4g-4+n},-1^n\big)$ that consists of isomorphism classes of pairs $(C,q)$, where $C$ is a smooth curve, and $q$ has exactly $4g-4+n$ simple zeros and $n$ simple poles. We call $\Qc\big(1^{4g-4+n},-1^n\big)$ the {\it principal stratum}.

Denote by $P\Mgno=\Mgno/\C^*$ the projectivization of~$\Mgno$, where $\C^*$ acts on quadratic differentials by multiplication. The complement $P\Mgno\setminus P\Qc\big(1^{4g-4+n},-1^n\big)$ is a union of divisors that we denote by $\D^0_{\deg}$, $\D^{\infty}_{\deg}$, and~$D_{\rm DM}$ (the subscript~$\deg$ stands for {\it degenerate}, as opposed to differentials in $\Qc\big(1^{4g-4+n},-1^n\big)$ that we call {\it generic}). The divisor $\D^0_{\deg}$ is the closure of the set of differentials $q$ with multiple zeros. The divisor $\D^\infty_{\deg}$ is the closure of $n$ strata $P\Qc\big(1^{4g-5+n},0,-1^{n-1}\big)$ (one for each pole) parameterizing differentials with $4g-5+n$ simple zeros, $n-1$ simple poles and one marked ordinary point (where a zero and a pole of a generic differential coalesce). Finally, $D_{\rm DM}$ is the pullback to $P\Mgno$ of the Deligne--Mumford boundary of~$\M_{g,n}$. The latter consists of the divisor $\D_{\rm irr}$ of irreducible nodal curves with $n$ marked points, and divisors $D_{j,k}$ parameterizing reducible curves with components of genera $j$ and $g-j$ having~$k$ and $n-k$ marked points respectively, where $j=0,\dots,[g/2]$, $k=0,\dots,n$ and $2<2j+k<2g+n-2$. (Note that while a quadratic differential may have only simple poles at the $n$ marked points, it may have poles up to order $n$ at the nodes.)

Denote by $L\to P\Mgno$ the tautological line bundle associated with the projectivization $\Mgno\to P\Mgno$,
and put $\phi=c_1(L)\in {\rm Pic}\big(P\Mgno\big)\otimes\Q$. Denote by $\lambda$ the pullback of the Hodge class $\lambda_1=\det\pi_*\omega_{g,n}$ on~$\Mb_{g,n}$.
Furthermore, denote by $\delta_{\deg}^0$, $\delta_{\deg}^\infty$, $\delta_{\rm irr}$, $\delta_{j,k}$ the classes in ${\rm Pic}\big(P\Mgno\big)\otimes\Q$ of the corresponding divisors. We will also use the standard notation~$\psi_i$, $i=1,\dots,n,$ for the tautological classes on~$\Mb_{g,n}$ as well as for their pullbacks to~$P\Mgno$.

Combining the results of \cite[Theorem 2]{AC} and, e.g., \cite[Theorem 3.3(b)]{Fu}, we get
\begin{Lemma}\label{pic}
The rational Picard group ${\rm Pic}\big(P\Mgno\big)\otimes\Q$ is freely generated over $\Q$ by the classes $\phi$, $\lambda$, $\psi_i$, $\delta_{\rm irr}$, $\delta_{j,k}$, where $i,k=1,\dots,n$, $j= 0,\dots, [g/2]$, and $2 < 2j+k < 2g+n-2$.
\end{Lemma}

To each pair $(C,q)$ one can canonically associate a twofold branched cover $f\colon \Ch\to C$ and an abelian differential~$v$ on~$\Ch$,
where
$\Ch=\big\{(x,v(x))\,|\,x\in C,\,v(x)\in T_x^*C,\,v(x)^2=q(x)\big\}$. For a~generic $(C,q)\in\Mgn$
the curve $\Ch$ is smooth of genus $\gh=4g-3+n$ and $v$ is holomorphic.

The covering $f$ is invariant under the canonical involution $(x,v(x))\mapsto (x,-v(x))$ on $\Ch$ that we denote by~$\mu$. Zeros and poles of $q$ of odd order are branch points
of the covering~$f$. The abelian differential~$v$ has second order zeros at simple zeros of~$q$, and at the simple poles of $q$ the differential $v$ is holomorphic and nonvanishing.

Consider the map $\ph\colon P\Mgno\to \Mb_{\gh}$, $(C,q)\mapsto \Ch$. This map induces a vector bundle $\ph^*\Lambda^1_{\gh}\to \Mgno$
of dimension $\gh=4g-3+n$, where $\Lambda^1_{\gh}\to \Mb_{\gh}$ is the Hodge vector bundle. The involution~$\mu$ on
$\Ch$ induces an involution~$\mu^*$ on the vector bundle $\ph^*\Lambda^1_{\gh}$. Hence we have a decomposition
\begin{gather*}
\ph^*\Lambda^1_{\gh}=\Lambda_+\oplus \Lambda_-,
\end{gather*}
where $\Lambda_+$ (resp. $ \Lambda_-$) is the eigenbundle corresponding to the eigenvalue $+1$ (resp.~$-1$) of~$\mu^*$.
Clearly, $\Lambda_+=p^*\pi_*\omega_g$ is the pullback of the Hodge bundle on $\Mb_g$, where $p\colon P\Mgno\to\Mb_g$ is a natural projection (forgetful map). We call $\Lambda_-$ the {\it Prym bundle}. Its fibers are the spaces of Prym differentials on $\Ch$ and have dimension $3g-3+n$. We call $\pt=c_1(\Lambda_-)\in {\rm Pic}\big(P\Mgno\big)\otimes\Q$ the {\it Prym
class} (slightly abusing the notation, we often denote the line bundles and their classes in the Picard group by the same symbols).

 In this paper we prove the following generalization of Theorem~1 in~\cite{Contemp} to the case $n>0$:
\begin{Theorem}\label{theoint}
The Hodge class $\lambda$ and the Prym class $\pt$ can be expressed in terms of the tautological class $\phi$ and the classes $\delta^0_{\deg}$, $\delta^\infty_{\deg}$ and $\delta_{\rm DM}$ by the formulas
\begin{gather*}
\l=\left(\f{5(g-1)}{36}-\f{n}{36}\right)\phi +\f{1}{72}\delta^0_{\deg}-\f{1}{18}{\delta}_{\deg}^\infty+\f{1}{12}\delta_{\rm DM},
%\label{l1int}
\\ %\label{lp1int}
\pt=\left(\f{11(g-1)}{36}+\f{5n}{36}\right)\phi +\f{13}{72}\delta^0_{\deg}+\f{5}{18}{\delta}_{\deg}^\infty+\f{1}{12}\delta_{\rm DM}.
\end{gather*}
\end{Theorem}

On the other hand, according to \cite[formula (5.16) and explanations thereafter]{Contemp},
\begin{gather}
\pt=\l_2-\f{1}{2}(3g-3+n)\phi,
\label{l2lp}
\end{gather}
where $\l_2=\det\pi_*\omega^2_{g,n}$. This implies
\begin{Corollary} The following relation holds in ${\rm Pic}\big(P\Mgno\big)\otimes\Q$:
\begin{gather}
\l_2-13\l=n\phi +\delta^\infty_{\deg}-\delta_{\rm DM}.
\label{l2l}
\end{gather}
\end{Corollary}

Furthermore, combining (\ref{l2l}) with Mumford's formula
\begin{gather}
\l_2-13\l=\sum_{i=1}^n \psi_i-\delta_{\rm DM}
\label{MumMgn}
\end{gather}
(cf., e.g., \cite{ArbCor}), formula (7.8)), we get the following
 \begin{Corollary} \label{deltapsi}
 The class $\delta^\infty_{\deg}$ is expressed via the sum of $\psi$-classes and class $\phi$ as follows:
 \begin{gather}
{\delta}_{\deg}^\infty = - n\phi+\sum_{i=1}^n \psi_i \,.
\label{delpsi}
\end{gather}
 \end{Corollary}

\begin{Remark}
An analytic approach to the classes $\psi_i$ was outlined, in particular, in~\cite{TZ}.
\end{Remark}

The paper is organized as follows. In Section~\ref{doubcov} we introduce a twofold canonical cover corresponding to
a pair consisting of a complex algebraic curve and a quadratic differential on it with $n$ simple poles. The main objective of
this section is to discuss the action of the covering involution on (co)homology of the cover and the associated matrix of $b$-periods. In Section~\ref{secdeftau} we define two tau functions corresponding to the eigenvalues $\pm 1$ of the covering map, discuss their basic properties and interpret them as holomorphic sections of line bundles on the moduli space of quadratic differentials with simple poles. In Section~\ref{secasymp} we study the asymptotic behavior of the tau functions near the divisor ${D}_{\deg}^\infty$ and use it to express the Hodge and Prym classes via the classes of the boundary divisors and the tautological class (the asymptotics of the tau functions near the divisors~${D}_{\deg}^0$ and~$D_{j,k}$ were thoroughly studied in~\cite{Contemp}).

\section{Geometry of the double cover}\label{doubcov}

Let $f\colon \Ch\to C$ be the double cover defined by the meromorphic quadratic differential $q$ with simple zeros $\{x_1,\dots,x_{4g-4+n}\}$ and simple poles $\{y_1,\dots,y_n\}$
on a smooth curve $C$, and let $\mu\colon \Ch\to\Ch$ be the corresponding involution.
The covering map $f$ is ramified over $x_j$, $y_k$ and we put $\hat{x}_j=f^{-1}(x_j)$, $\hat{y}_k=f^{-1}(y_k)$; the points $\hat{x}_j$, $\hat{y}_k$ are exactly
 the fixed points of~$\mu$. By $\mu_*$ (resp.~$\mu^*$) we denote the involution induced by $\mu$ in homology (resp.\ in cohomology) of~$\Ch$.
The space~$\Lambda^1_{\Ch}$ of holomorphic abelian differentials on $\Ch$ splits into two eigenspaces~$\Lambda_+$ and~$\Lambda_-$ of complex dimension $g$ and $3g-3+n$ respectively that correspond to the eigenvalues~$\pm 1$ of~$\mu^*$. We have a similar decomposition in the real homology of $\Ch$: $H_1\big(\Ch,\R\big)=H_+\oplus H_-$, where $\dim H_+=2g$, $\dim H_-=6g-6+2n$.
Following~\cite{Fay73}, we pick $8g-6+2n$ smooth 1-cycles on~$\Ch$
\begin{gather}
\big\{a_j, a_j^*, \at_k,b_j, b_j^*,\bt_k\big\} ,\qquad j=1,\dots,g,\quad k=1,\dots,2g+n-3,
\label{mainbasis}
\end{gather}
in such a way that
\[
\mu_* a_j = a_j^*,\qquad\mu_* b_j = b_j^*,
\qquad\mu_* \at_k + \at_k= \mu_* \bt_k + \bt_k = 0,
\]
and the intersection matrix is
\[
\left(\begin{matrix} 0 & I_{4g-3+n}\\- I_{4g-3+n} & 0 \end{matrix}\right)
\]
(here $I_k$ denotes the $k\times k$ identity matrix).

The projections $f_*a_j=f_*(a_j^*)$ and $f_*b_j=f_*(b_j^*)$ on $C$ form a canonical basis in $H_1(C)$, while the projections
$f_*\at_k$ and $f_*\bt_k$ are trivial in~$H_1(C)$.

\begin{Remark}To avoid complicated notation, we will use the same symbols for the cycles~(\ref{mainbasis}), their homology classes
in $H_1\big(\Ch\big)$, and their pushforwards in~$H_1(C)$. In particular, the images of~(\ref{mainbasis}) give rise to a~canonical basis in~$H_1(C)$.
\end{Remark}

Denote by $\{u_j, u^*_j ,\ut_k \}$ the basis of normalized abelian differentials on $\Ch$ associated with~(\ref{mainbasis}), so that the action of $\mu^*$ on
$\Lambda^1_{\Ch}$ is given by the matrix
\begin{gather*}
M= \left(\begin{matrix} 0 & I_g &0 \\
 I_g & 0 & 0 \\
 0 & 0 & -I_{2g-3+n} \end{matrix}\right) .
%\label{Smu}
\end{gather*}
The differentials $u_j^+=u_j+u_j^*$, $j=1,\dots,g$,
provide a basis in the space $\Lambda_+$, whereas a basis in~$\Lambda_-$ is given by $3g-3+n$ Prym differentials $u_l^-$, where
\begin{gather*}
u_l^-=\begin{cases}u_{l}-u_{l}^*,& l=1,\dots, g,\\
\ut_{l-g},& l=g+1,\dots, 3g-3+n.\end{cases}
%\label{Prym}
\end{gather*}

We also introduce the bases in the spaces $H_+$ and $H_-$.
The classes
\begin{gather*}
\a_j^+ = \f{1}{2}(a_j+ a_j^*) ,\qquad
\b_j^+ = b_j+ b_j^* ,\qquad j=1,\dots,g,
%\label{abp}
\end{gather*}
form a symplectic basis in $H_+$, whereas the classes
\begin{gather*}
\a_l^- = \f{1}{2}(a_{l}- a_{l}^*) ,\qquad \b_l^- = b_{l}- b_{l}^*,\qquad l=1,\dots,g ,\nonumber\\
\a^-_{l}=\at_{l-g},\qquad \b^-_{l}=\bt_{l-g} , \qquad l=g+1,\dots,3g-3+n
%\label{abm}
\end{gather*}
form a symplectic basis in $H_-$.
The basis $\{a_j^+,\a_l^-,\b_j^+,\b_l^-\}$, $j=1,\dots,g$, $l=1,\dots,3g-3+n$, is related to the canonical basis (\ref{mainbasis}) by means of a (non-integer) symplectic matrix
\begin{gather*}
S= \left(\begin{matrix} T & 0 \\
 0 & \big(T^t\big)^{-1}
 \end{matrix}\right)
%\label{mainmat}
\end{gather*}
with
\begin{gather}
T= \left(\begin{matrix} I_g & I_g& 0 \\
 I_g & -I_g & 0 \\
 0& 0 & I_{2g-3+n} \end{matrix}\right)\;.
%\label{MaMb}
\end{gather}
The differentials $u_j^+$, $u_l^-$ are normalized relative to the classes $\a_j^+$, $\a_l^-$ respectfully in the sense that
\[
\int_{\a_i^+} u_j^+=\delta_{ij} , \qquad \int_{\a_k^-} u_l^-=\delta_{kl},\]
where $i,j=1,\dots,g$ and $k,l=1,\dots,3g-3+n$.
The corresponding matrices of $\beta$-periods $\Omega_+$ and $\Omega_-$ are given by
\begin{gather*}
(\Omega_+)_{ij}=\int_{\beta_i^+} u_j^+,\qquad i,j=1,\dots, g,\\
%\label{sigp}\\
(\Omega_-)_{kl}=\int_{\beta_k^-} u_l^- , \qquad k,l=1,\dots, 3g-3+n .
%\label{sigm}
\end{gather*}
 The matrix $\widehat{\Omega}$ of $b$-periods of $\{u_j, \ut_k, u^*_j \}$ with respect to the homology basis~(\ref{mainbasis}) on $\Ch$ is related to $\Omega^+$ and $\Omega^-$ by the formula
\begin{gather*}
\widehat{\Omega}=T^{-1}
 \left(\begin{matrix} \Omega_+ & 0\\
 0 & \Omega_- \end{matrix}\right) \big(T^t\big)^{-1}.
%\label{shs}
\end{gather*}

We proceed with bidifferentials and projective connections on the double covers.
Let $\Bh(x,y)$ denote the canonical (Bergman) bidifferential on $\Ch\times\Ch$ associated with the homology basis~(\ref{mainbasis}); $\Bh(x,y)$ is symmetric, has the second order pole on the diagonal $x=y$ with biresidue~1, and all $a$-periods of~$\Bh(x,y)$
on~$\Ch$ vanish. We put
 \begin{gather*}
B_+(x,y)=\Bh(x,y)+\mu_y^*\Bh(x,y), \qquad B_-(x,y)=\Bh(x,y)-\mu_y^*\Bh(x,y),
%\label{defBpm}
\end{gather*}
(the subscript $y$ at $\mu^*$ means that we take the pullback with respect to the involution on the second factor in $\Ch\times\Ch$).
The bidifferential $B_+(x,y)$ is the pullback of the canonical bidifferential $B(x,y)$ on $C\times C$ (normalized relative to the classes $f_* a_j$, where $f\colon \Ch\to C$ is the covering map). The bidifferential $B_-(x,y)$ is called {\it Prym bidifferential} in~\cite{Contemp}.

Properties of the differentials $B_\pm$ are summarized in \cite[Lemma 2]{Contemp}.
Near the diagonal $x=y$ in $\Ch\times\Ch$ we have
\begin{gather*}
B_\pm(x,y)= \f{{\rm d}\zeta(x){\rm d}\zeta(y)}{(\zeta(x)-\zeta(y))^2}+\f{1}{6}S_{B_{\pm}}(\zeta(x))
+\cdots,
%\label{Bpmdiag}
\end{gather*}
where $\zeta(x)$ is a local parameter at $x\in\Ch$, and
 two projective connections $S_{B_+}$ and $S_{B_-}$ that are related by
\begin{gather}
S_{B_\pm}(x)=S_{\Bh}(x)\pm 6 \mu_y^*\Bh(x,y)\big|_{y=x}.\label{defSbpm}
\end{gather}
Here $S_{B_-}$ is the {\it Prym projective connection}.

Now we describe the dependence of bidifferentials and projective connections on the choice of homology basis.
Let $\sigma$ be a symplectic (i.e., preserving the intersection form) transformation of $H_1\big(\Ch,\Z\big)$. In the canonical basis~(\ref{mainbasis}), $\sigma$ is represented by an ${\rm Sp}(8g-6+2n,\Z)$-matrix
$\left(\begin{smallmatrix} A & B\\C & D \end{smallmatrix}\right)$ with square blocks of size $4g-3+n$. The canonical bidifferential $\Bh^\sigma(x,y)$ on $\Ch\times\Ch$ associated with the new basis is
\begin{gather*}
\Bh^\sigma (x,y) = \Bh(x,y) -2\pi \sqrt{-1} \bu(x)^t \big(C\wO+ D\big)^{-1} C \bu(y),
%\label{transB}
\end{gather*}
where ${\bu}=\{u_j,\tilde{u_k},u_j^*\}^t$ is the vector of holomorphic abelian differentials on $\Ch$ normalized with respect to $\{a_j,\tilde{a_k},a_j^*\}$, $
j=1,\dots,g$, $k=1,\dots, 2g-3+n$.

If $\sigma$ commutes with the involution $\mu_*$, then $\sigma=\operatorname{diag}(\sigma_+,\sigma_-)$, where $\sigma_+$ (resp.~$\sigma_-$) is a symplectic transformation of $H_+$ (resp.~$H_-$) with half-integer coefficients (note that the intersection form on $H_1\big(\Ch,\R\big)$ is invariant with respect to~$\mu_*$). In the bases $\{\a_j^+,\b_j^+\}$, $j=1,\dots,g$, and $\{\a_l^-,\b_l^-\}$, $l=1,\dots, 3g-3+n$, these transformations can be written as
\begin{gather}
\sigma_\pm:= \left(\begin{matrix} A_\pm & B_\pm\\
 C_\pm & D_\pm \end{matrix}\right).
\label{Qpm}
\end{gather}

Transformation properties of basic holomorphic differentials, Prym matrix, differentials $B_\pm$ and projective connections ${S_B}_\pm$ are described in \cite[Lemma~3]{Contemp}.

\section{Tau functions}\label{secdeftau}

Here we define the necessary tau functions and study their basic properties.
Only slight modifications are needed comparing to the case $n=0$~\cite{Contemp}.

\subsection[Definition of tau\_\{pm\}]{Definition of $\boldsymbol{\tau_\pm}$}

Consider two meromorphic quadratic differentials $S_{B_\pm}-S_v$ on $\Ch$, where $S_{B_\pm}$ are the Bergman projective connections~(\ref{defSbpm}) and $S_v$ is given by
 \begin{gather*}
S_v=\left(\frac{v'}{v}\right)'-\frac{1}{2}\left(\frac{v'}{v}\right)^2.%\label{defSv}
\end{gather*}
where the prime means the derivative with respect to a local parameter $\zeta(x)$ (in other words, $S_v$ is the Schwarzian derivative of the abelian integral $\int_{x_0}^x v$).
Take the trivial line bundle on the principal stratum $\mathcal{Q}\big(1^{4g-4+n}, -1^n\big)$ and consider two connections
\begin{gather*}
d_\pm=d+\qb_\pm%\label{conn}
\end{gather*}
with
\begin{gather*}
\qb_\pm=-\sum_{i=1}^{6g-6+2n} \left(\int_{s_i^*} \varphi_\pm\right)d\int_{s_i} v,
%\label{qbk}
\end{gather*}
where
$\{s_i\}=\{\a_i,\b_i\}$ and $\{s_i^*\}=\{\b_i, -\a_i\}$, $i=1,\dots,3g-3+n$, being the dual bases in $H_-$ as above,
and the meromorphic abelian differentials~$\varphi_\pm$ are given by
\begin{gather*}
\varphi_\pm=-\f{2}{\pi\sqrt{-1}}\,\f{ S_{B_{\pm}}-S_v}{v} .
%\label{defOk}
\end{gather*}
The connections $d_\pm$ are flat, see \cite{JDG,review,Contemp} for details.

\begin{Definition}
The tau functions $\tau_\pm$ are (locally) covariant constant sections of the trivial line bundle on the principal stratum
$\mathcal{Q}\big(1^{4g-4+n}, -1^n\big)$ with respect to the connections~$d_\pm$, that is,
\begin{gather}
d_\pm \tau_\pm=0.
\label{deftau}
\end{gather}
\end{Definition}

Explicit formulas for the tau functions $\tau_\pm$ can be derived similar
to~\cite{JDG,Contemp}, but we are not going to use them here.

\begin{Remark}\label{rem1}
The definition (\ref{deftau}) follows the convention of \cite{Contemp}. The tau functions defined by~(\ref{deftau})
are equal to the $48$th power of the tau functions defined in~\cite{JDG, review}.
\end{Remark}

\subsection[Transformation properties of tau\_\{pm\}]{Transformation properties of $\boldsymbol{\tau_\pm}$}

For our purposes we need to analyze transformation properties of $\tau_\pm$.
The group $\C^*$ of nonzero complex numbers acts by multiplication on quadratic differential $q$ on $C$
and, therefore, on the tau functions $\tau_\pm$.
Under the action of $\e\in\C^*$ the tau functions $\tau_\pm$ transform like $\tau_\pm(C,\e q)= \e^{\kappa_\pm} \tau_\pm(C, q)$, where
\begin{gather}
\kappa_+=\f{20(g-1)-4n}{3} , \qquad \kappa_-=\f{44(g-1)+20n}{3}.
\label{ppm}
\end{gather}
The formulas (\ref{ppm}) follow from formulas (6.54) and (6.55) of~\cite{review} taking into account an extra factor of $48$ due to a different definition of $\tau_\pm$ adopted here, see also~(\ref{kakh}).

Let us now describe the behavior of the tau functions under a change of the canonical homology basis that commutes with the action of the involution $\mu$. Note that~$\tau_+$ (resp.~$\tau_-$) is uniquely determined by the subspace in~$H_+$ (resp.~$H_-$) generated by the classes $\a_i^+$ (resp.~$\a_j^-$).

Let $\s$ be a symplectic transformation of $H_1\big(\Ch,\R\big)$ commuting with $\mu_*$ that is given by the matrices $\s_\pm$ in the basis $\{\a_i^+,\b_i^+,\a_j^-,\b_j^-\}$, $i=1,\dots,g$, $j=1,\dots,3g-3+n$, as in~(\ref{Qpm}).
Then the tau functions $\tau_\pm$ transform under the action of $\s$ by the formula
\begin{gather}
 \f{\tau_\pm^{\s}}{\tau_\pm}= \gamma_\pm(\s_\pm) \det (C_\pm\Omega_\pm + D_\pm)^{48},
\label{transtau3}
\end{gather}
where $\gamma_\pm(\s_\pm)^3 = 1$.

The proof of this statement for $n>0$ is similar to the case $n=0$
considered in \cite[Theorem~3]{Contemp}. Namely, the transformation property (\ref{transtau3}) (with some
constants $\gamma_\pm$) follows from the differential equations
(\ref{deftau}) and the transformation rule of the Bergman projective connections ${S_B}_\pm$ under the change of canonical bases in~$H_\pm$. The proof that $\gamma_\pm^3=1$ uses explicit formulas for $\tau_\pm$ which use the distinguished local parameters near the ramification points of the cover $C\to\Ch$. The details can be found in~\cite{review},
see formulas~(6.46) and~(6.50) therein.

The following statement is an immediate consequence of (\ref{ppm}) and (\ref{transtau3}):

\begin{Theorem}\label{sect}
For the tau function $\tau_+$ $($resp.~$\tau_-)$ its $3$rd power $\tau_+^3$ $\big($resp.~$\tau_-^3\big)$ is a nowhere vanishing holomorphic section of the line bundle $\lambda^{144}\otimes L^{{20(g-1)-4n}}$ $\big($resp.\ $\pt^{144}\otimes L^{44(g-1)+20n}\big)$
on the $($projectivized$)$ principal stratum $P\mathcal{Q}\big(1^{4g-4+n}, -1^n\big)$, where~$L$ denotes the tautological line bundle on $P\mathcal{Q}\big(1^{4g-4+n}, -1^n\big)$.
\end{Theorem}

\begin{Corollary}
The Hodge and Prym classes in the rational Picard group of $P\mathcal{Q}\big(1^{4g-4+n}, -1^n\big)$ satisfy the relations
\begin{gather*}
\lambda=\f{5(g-1)-n}{36} \phi ,\qquad \pt=\f{11(g-1)+5n}{36} \phi ,
%\label{PTnon-compp}
\end{gather*}
where $\phi=c_1(L)$.
\end{Corollary}

\section[The Hodge and Prym classes on PQ\_\{g,n\}]{The Hodge and Prym classes on $\boldsymbol{P\Mgno}$}\label{secasymp}

Here we compute the divisor classes of the tau functions $\tau_\pm$ viewed as holomorphic sections of line bundles on the compactification~$P\Mgno$. The
boundary $P\Mgno\setminus P\mathcal{Q}\big(1^{4g-4+n}, -1^n\big)$ consists of divisors
 $D^0_{\deg}$, $\D^{\infty}_{\deg}$ and $D_{\rm DM}$. The asymptotic behaviour of $\tau_\pm$ near $D^0_{\deg}$ and $D_{\rm DM}$ for $n>0$ is the same as for $n=0$, the case considered in detail in~\cite{Contemp}.

 Namely, in terms of transversal local coordinate $t$ near $D_{\rm DM}$ we have{\samepage
 \begin{gather}
\tau_\pm = t^{4}(\text{const}+ o(1))\qquad {\rm as} \ t\to 0,
\label{tauDM}
\end{gather}
cf.\ formulas (5.10) and (5.12) in~\cite{Contemp}.}

In terms of a local transversal coordinate $t$ near $D^0_{\deg}$
 \begin{gather}
\tau_+ = t^{2/3}(\text{const}+ o(1)) ,\qquad \tau_- = t^{26/3}(\text{const}+ o(1))
\label{d1}
\end{gather}
as $t\to 0$, cf.\ Lemmas~8 and~9 in \cite{Contemp}.

 Therefore, it remains to analyze the asymptotics of $\tau_\pm$ near the divisor
 $\D^\infty_{\deg}$.

\subsection[Local behavior of tau functions near D\textasciicircum{}\{infty\}\_\{deg\}]{Local behavior of tau functions near $\boldsymbol{\D^{\infty}_{\deg}}$}

 Recall that the boundary component $D^\infty_{\deg}$ is the closure in $P\Mgno$ of the union of~$n$ copies of the set ${\mathcal Q}\big(1^{4g-5+n},0,-1^{n-1}\big)$ of equivalence classes of pairs $(C,q)$, where $C$ is a smooth curve with a marked point and $q$ is a quadratic differential with $n-1$ labeled simple poles and $4g-5+n$ simple zeros (the marked point is neither zero nor pole of~$q$ and distinguish the point where a~zero and a~pole of $q$ coalesce).

Consider a one parameter family $(C_t,q_t)$ of generic quadratic differentials transversal to $D^\infty_{\deg}$ that converges to $(C_0,q_0)\in D^\infty_{\deg}$ as $t\to 0$ (we can actually keep $C_t$ fixed and ignore the rational tail of $C_0$). Without loss of generality we can assume that under such a degeneration a simple zero $x_{1}^t$ and a simple pole $y_1^t$ of $q_t$ on $C_t$ coalesce as $t\to 0$. Denote by $s_t\in H_-$ the vanishing cycle on $\Ch$ going around the path connecting the preimages $\hat{x}_{1}^t, \hat{y}_1^t\in\Ch$, and denote by $t=\int_{s_t}v_t$ the corresponding homological coordinate.

Assuming that the zero $x_{1}^t$ and the pole $y_1^t$ are close to each other, consider a small disk $U_t\subset C_t$ that contains $x_{1}^t$, $y_1^t$ and no other singularities of~$q_t$.

\begin{Lemma}
The local coordinate $t=\int_{s_t} v_t$ is transversal to $D^\infty_{\deg}$ in a tubular neighborhood of $D^\infty_{\deg}\subset\Mgno$.
\end{Lemma}
\begin{proof}
The proof is similar to that of Lemma~8 in~\cite{Contemp}. Namely, one can choose a coordinate~$\zeta$ in~$U_t$ such that
\begin{gather*}
q(\zeta)= \f{\zeta-\zeta\big(x_{1}^t\big)}{\zeta-\zeta\big(y_1^t\big)} {\rm d}\zeta^2 .
\end{gather*}
Then
\begin{gather*}
t=\int_{s_t} v_t=2\int_{\zeta(x_{1}^t)}^{\zeta(y_1^t)}\left(\f{\zeta-\zeta\big(x_{1}^t\big)}{\zeta-\zeta\big(y_1^t\big)}\right)^{1/2}{\rm d}\zeta
={\rm const}\big( \zeta\big(x_{1}^t\big)-\zeta\big(y_1^t\big)\big).
\end{gather*}
Therefore, $t=\int_{s_t} v_t$ is transversal to $D^\infty_{\deg}$.
\end{proof}

\begin{Lemma}
The tau functions $\tau_\pm (C_t,q_t)$ have the following asymptotic behaviour as $t\to 0$:
\begin{gather}
\tau_+(C_t,q_t)\sim t^{-8/3}\tau_+(C_0,q_0),\qquad
\tau_-(C_t,q_t) \sim t^{40/3}\tau_-(C_0,q_0).\label{Dt}
\end{gather}
\end{Lemma}

\begin{proof} First, let us notice that for an arbitrary stratum ${\mathcal Q}(d_1,\dots,d_k)$, where $d_1,\dots,d_k$ are integers not less than $-1$ and $\sum_{i=1}^k d_i = 4g-4$, the tau functions $\tau_\pm$ have the following homogeneity property:
$\tau_\pm(C,\e q)=\e^{\kappa_{\pm}} \tau_\pm(C,q)$ with
\begin{gather}
\kappa_+= \sum_{i}\f{d_i(d_i+4)}{d_i+2} , \qquad
\kappa_-= \kappa_++6\sum_{d_i \, {\rm odd}} \f{1}{d_i+2} ,
\label{kakh}
\end{gather}
 (see equations~(6.54) and (6.55) of \cite{review}, as well as Remark~\ref{rem1}). Now, in the limit $t\to 0$, we have for some~$\gamma_\pm$
 \begin{gather*}
 \tau_\pm\sim t^{\gamma_\pm}\tau_\pm(C_0,q_0).
\end{gather*}
 To compute $\gamma_\pm$ we evaluate the differences of the homogeneity coefficients (\ref{kakh})
 on the strata ${\mathcal Q}\big(1^{4g-4+n},-1^{n}\big)$ and ${\mathcal Q}\big(1^{4g-5+n},0,-1^{n}\big)$,
 i.e., when a pole and a zero annihilate each other (all other contributions remain the same):
 \begin{gather*}
 \kappa_+(q)-\kappa_+(q_0)=\f{1\cdot 5}{3}+ \f{(-1)\cdot 3}{1}=-\f{4}{3},\\
 \kappa_-(q)-\kappa_-(q_0)=\kappa_+(q)-\kappa_+(q_0) +6\left(\f{1}{3}+1\right)=\f{20}{3}.
\end{gather*}
Since the homogeneity coefficient of the coordinate $t$ is 1/2,
we have $\gamma_+=-8/3$ and $\gamma_-=40/3$.
\end{proof}

\subsection{Proof of Theorem~\ref{theoint} and its consequences}
Combining Theorem \ref{sect} with the asymptotics~(\ref{tauDM}), (\ref{d1}) and~(\ref{Dt}) of $\tau_\pm$ near~$D_{\rm DM}$, $D^0_{\deg}$ and~$\D^\infty_{\deg}$, we get the formulas
\begin{gather}
48 \l-\f{20(g-1)-4n}{3} \phi=\f{2}{3} \delta^0_{\deg}-\f{8}{3}{\delta}^\infty_{\deg}+4\delta_{\rm DM} ,
\label{PTMbarp}\\
48 \pt -\f{44(g-1)+20n}{3} \phi=\f{26}{3} \delta^0_{\deg}+\f{40}{3}{\delta}^\infty_{\deg} +4\delta_{\rm DM},
\label{PTMbarm}
\end{gather}
which immediately imply Theorem~\ref{theoint}.

Excluding $\delta^0_{\deg}$ from (\ref{PTMbarp}) and (\ref{PTMbarm}), we obtain
\begin{Corollary}
The Prym class $\pt$ on $P\Mgno$ is expressed in terms of the Hodge class~$\l$, the tautological class~$\phi$ and
the boundary classes $D_{\rm DM}$, $D^\infty_{\deg}$ by the formula
\begin{gather*}
\pt- 13 \l ={\delta}^\infty_{\deg} - \delta_{\rm DM} -\f{1}{2}(3g-3-n)\phi .
%\label{lptl}
\end{gather*}
\end{Corollary}

Furthermore, using the formula $\pt=\l_2-\f{1}{2}(3g-3+n)\phi$ that relates the Prym class to the class $\lambda_2=\pi_*\omega^2_{g,n}$ (see (\ref{l2lp})), we get

\begin{Corollary}
The following relation holds:
\begin{gather*}
\l_2 - 13 \l =n\phi + {\delta}^\infty_{\deg} - \delta_{\rm DM} .
\end{gather*}
\end{Corollary}

Combining this relation with the pullback to $P\Mgno$ of Mumford's formula~(\ref{MumMgn}), we obtain the equation~(\ref{delpsi})
 \begin{gather*}
 {\delta}^\infty_{\deg} = - n\phi +\sum_{i=1}^n \psi_i.
 \end{gather*}

 On the other hand, excluding $ {\delta}^\infty_{\deg} $ from (\ref{PTMbarp}) and
 (\ref{PTMbarm}), we get the formula
 \begin{gather*}
 \delta^0_{\deg}= 72\lambda + 4\sum_{i=1}^n \psi_i - (10(g-1)+2n)\phi -6\delta_{\rm DM}
 \end{gather*}
 that coincides with the formula~(1.2) in \cite{Contemp} for $n=0$.

\subsection*{Acknowledgements}
This work was supported by the Ministry of Science and Higher Education of the Russian Federation, agreement no.~075-15-2019-1620. We thank A.~Zorich for useful discussions. We are grateful to anonymous referees for carefully reading the manuscript and pointing out several typos and other inconsistencies.

\pdfbookmark[1]{References}{ref}
\LastPageEnding

\end{document}